\documentclass[preprint]{elsarticle}

\usepackage{amsfonts,amssymb,amsmath}

\newenvironment{proof}{\medskip                    %% Proof
\noindent{\scshape Proof:}}{\quad $\square$
\medskip}  %%

\usepackage{hyperref}

\newtheorem{theorem}{Theorem}[section]
\newtheorem{lemma}[theorem]{Lemma}
\newtheorem{proposition}[theorem]{Proposition}
\newtheorem{corollary}[theorem]{Corollary}

\newtheorem{remark}[theorem]{Remark}

\newtheorem{definition}[theorem]{Definition}

\newfont{\bb}{msbm10}

\def\diag{ {\rm diag}}

\def\b1{{\bf 1}}

\newcommand{\Aux}{\operatorname{Aux}}
\newcommand {\beq}{\begin{equation}}%\label{#1}}
\newcommand {\eeq} {\end{equation}}
\newcommand {\cC} {{\cal C}}

\newcommand {\cG} {{\cal G}}

\newcommand {\cS} {{\cal S}}

\newcommand{\digr}{\cG}
\newcommand{\crit}{\cC}
\newcommand{\cycles}{C}
\newcommand {\R} {{\mathbb R}}

\newcommand {\Rp} {\R_+}
\newcommand {\Rpn} {\R_+^n}
\newcommand {\Rpnn} {\Rp^{n\times n}}

\newcommand{\comp}{\operatorname{comp}}

\def\C{{\rm C\kern-.48em\vrule width.06em height.6em depth-.02em
                  \kern.48em}}
\def\bce{\begin{center}}
\def\ece{\end{center}}

\begin{document}

\title{On sets of eigenvalues of matrices with prescribed row sums
and prescribed graph } 

\author[rvt1]{Gernot Michael Engel\corref{cor}}%\fnref{fn1}}
\ead{engel@transversalnetworks.net}

\author[rvt2]{Hans Schneider}%\fnref{fn1}}
\ead{hans@math.wisc.edu}

\author[rvt3]{Serge{\u\i} Sergeev\fnref{fn1}}
\ead{sergiej@gmail.com}

\address[rvt1]{Transversal Networks Corp., 2753 Mashall Parkway, Madison, WI 53713,USA}
\address[rvt2]{Deaprtment of Mathematics, University of Wisconsin-Madison, Madison, WI 53706, USA}
\address[rvt3]{University of Birmingham, School of Mathematics, Edgbaston B15 2TT, UK}

\cortext[cor]{Corresponding author. Email: engel@transversalnetworks.net}
%\fntext[fn1]{Supported by the Czech Science Foundation project \# 402/09/0405 and Grant Agency of
%Excellence UHK FIM \# 2214.}

\fntext[fn1]{Supported by EPSRC grant EP/J00829X/1}

%%%%%%%%%%%%%%%%%%%%%%%%%%%%%%%%%%%%%%%%%%%%%%%%%%%%%%%%%%%%%%%%%%%%%%%%
%%%                         ABSTRACT
%%%%%%%%%%%%%%%%%%%%%%%%%%%%%%%%%%%%%%%%%%%%%%%%%%%%%%%%%%%%%%%%%%%%%%%%
\begin{abstract}
Motivated by a work of Boros, Brualdi, Crama and Hoffman, we consider the sets of 
(i) possible Perron roots of nonnegative matrices with prescribed
row sums and associated graph, and (ii) possible eigenvalues of complex matrices with 
prescribed associated graph and
row sums of the moduli of their entries.
To characterize the set of Perron roots or possible eigenvalues of matrices in these 
classes we introduce, following an idea of Al'pin,
Elsner and van den Driessche, the concept of row uniform matrix, which is a nonnegative matrix where all nonzero
entries in every row are equal.
% Extending known results
%to the reducible case, we derive new bounds on the set of 
%eigenvalues or Perron roots of matrices in a given class of nonnegative or complex matrices. 
%in terms of maximal and minimal cycle geometric means of irreducible blocks. 
Furthermore, we 
completely characterize the sets of possible Perron roots of the class of nonnegative matrices 
and the set of possible eigenvalues of the class of complex matrices under study.
% "`we completely characterize" is a overstatement. I could see where someone in the future could prove a result
% that bounds the absolute value of the possible product of the eigenvalues thus further characterizing the possible 
% eigenvalues for the same class of complex matricies we have studied. 
Extending known results
to the reducible case, we derive new sharp bounds on the set of 
eigenvalues or Perron roots of matrices when the only information 
available is the graph of the matrix and the row sums of the moduli of its entries. In the last section of the paper
a new constructive proof of the Camion-Hoffman theorem is given. 
\end{abstract}

%%%%%%%%%%%%%%%%%%%%%%%%%%%%%%%%%%%%%%%%%%%%%%%%%%%%%%%%%%%%%%%%%%%%%%%%
%%%                         KEYWORDS
%%%%%%%%%%%%%%%%%%%%%%%%%%%%%%%%%%%%%%%%%%%%%%%%%%%%%%%%%%%%%%%%%%%%%%%%
\begin{keyword}
 Ger{\v{s}}gorin, eigenvalues, Perron root, row sums, row uniform matrices, graphs, diagonal similarity,  sum scaling, Camion-Hoffman, 
\vskip0.1cm {\it{AMS Classification:}} 15A18, 15A29, 15A80
% 15A18: Eigenvalues, singular values and eigenvectors.
%15A29: Inverse problems
%15A80: Max-plus and related algebras
\end{keyword}

\maketitle

%\title{On the sets of eigenvalues of matrices with prescribed
% row sums and prescribed graph\\ \version}
% SUGGESTED CHANGE

%\tnotetext[label1]{This work is partially supported by
%RFBR-CRNF grant 11-01-93106 and EPSRC RRAH15735}


\section{Introduction}

\subsection{Background and motivation}

The use of the row sums of a  matrix to determine 
nonsingularity or to bound 
its spectrum has its origins in the 19th century \cite[Section 2] {S} and has led 
to a vast literature associated with the name of Ger\v{s}gorin and his circles \cite{V}. One of the first observations,
due to
\if{
The main idea behind the present paper is  by simple expressions involving
row sums of that matrix. This idea, closely related to Ger\v{s}gorin circles was known to
} \fi 
Frobenius, was that the Perron root $\rho(A)$ (i.e., the biggest nonnegative 
eigenvalue, or the spectral radius) of a nonnegative matrix $A\in\Rpnn$ is bounded by 
%bound the Perron root $\rho(A)$ for $A\in\Rpnn$ by
\begin{equation}
\label{e:rhoA}
\min\limits_{i=1}^n r_i(A)\leq\rho(A)\leq\max_{i=1}^n r_i(A)
\end{equation}
where $r_i$ denotes the $i$th row sum of the elements of $A$. If $A$ is irreducible
then the inequalities in~\eqref{e:rhoA} are strict 
except when $\min_{i=1}^n r_i(A)=\max_{i=1}^n r_i(A)$. 

In a recent development, Al'pin~\cite{A},
Elsner and van den Driessche~\cite{EvdD} sharpened the classical bounds of Frobenius
by considering a matrix $B$ which has the same zero-nonzero pattern
as $A$, %denoted by $\digr(B)(=\digr(A))$ 
and whose entries are equal to the row sums of $A$ in the corresponding rows. 
We formalize this idea in the following definition.

%We begin by defining an auxiliary matrix used in
% \cite{EvdD} where its source is traced back to \cite{A} and (in a slightly different form)  in \cite{BBCH}.

\begin{definition} 
\label{def:aux}
For  $A \in \Rpnn$ we define the {\em auxiliary matrix} $B = \Aux(A)$ defined by 
\begin{equation}
\label{auxmat}
\begin{cases}
b_{ij} = \sum_k a_{ik}, &  \text{if   $a_{ij} \neq 0$},\\
b_{ij} = 0,             & \text{if    $a_{ij} = 0$}.
\end{cases}
\end{equation}
For a general complex matrix $A\in\C^{n\times n}$, its auxiliary matrix is defined as
$\Aux(|A|)$.
\end{definition}

Next, recall the concepts of minimal and maximal cycle (geometric) means. For an arbitrary 
matrix $A\in\Rpnn$ these quantities are defined as follows 

\begin{equation}
\label{e:numu}
\begin{split}
\nu(A)&= \min\limits_{(i_1,\ldots, i_{\ell})\in \cycles(A)}
(a_{i_1i_2}\cdot a_{i_2i_3}\cdot\ldots\cdot a_{i_{\ell}i_1})^{1/\ell},\\
\mu(A)&= \max\limits_{(i_1,\ldots, i_{\ell})\in \cycles(A)}
(a_{i_1i_2}\cdot a_{i_2i_3}\cdot\ldots\cdot a_{i_{\ell}i_1})^{1/\ell},
\end{split}
\end{equation}
where $\cycles(A)$ denotes the set of cycles of the associated graph. Recall that the directed weighted graph, associated with an arbitrary complex matrix $A\in\C^{nn}$, is 
defined by the set of nodes $N=\{1,\ldots,n\}$ and set of edges $E$ such that $(i,j)\in E$ if and only if
$a_{ij}\neq 0$, in which case edge $(i,j)$ is assigned the weight $a_{ij}$. 

According to Al'pin~\cite{A}, Elsner and van den Driessche~\cite{EvdD}, we have
\begin{equation}
\label{e:nurhomu}
\nu(B)\leq\rho(A)\leq\mu(B),\qquad B=\Aux(A),
\end{equation}
for any nonnegative matrix $A$.
If $A$ and hence $B$ are irreducible then either $\nu(B)=\rho(A)=\mu(B)$ or (if $\nu(B)<\mu(B)$) the inequalities 
in~\eqref{e:nurhomu} are strict. 

Exploiting similar ideas, Boros, Brualdi, Crama and Hoffman~\cite{BBCH} investigated a 
class of complex matrices $A\in\C^{n\times n}$ with 
prescribed off-diagonal row sums of the moduli of their entries, prescribed associated graph, and prescribed
moduli of all diagonal entries. In the case when $\digr(A)$ is {\bf s}trongly {\bf c}onnected {\bf w}ith {\bf a}t 
{\bf l}east {\bf t}wo {\bf cy}cles (scwaltcy), they investigated the existence of a 
positive vector $x$ satisfying
\begin{equation}
\label{e:bbch}
 |a_{ii}| x_i\geq \sum_{j\neq i} |a_{ij}| x_j,\quad i=1,\ldots,n
\end{equation}
for all matrices from the class simultaneously, and described the cases when all inequalities
in~\eqref{e:bbch} are strict~\cite[Theorem 1.1]{BBCH}, at least one of 
the inequalities is strict~\cite[Theorem 1.2]{BBCH}, or all 
inequalities hold with an equality~\cite[Theorem 1.3]{BBCH}. These results imply 
generalizations of Ger\v{s}gorin's theorem  due to Brualdi \cite{B}.
 Following the statement of~\cite[Theorem 1.4]{BBCH} the authors provide a detailed
outline for the proof that Brualdi's conditions are sharp.

In this paper we mainly deal with the two classes of matrices 
described in the abstract.  These classes are similar to those in~\cite{BBCH}, 
but we drop the requirement that 
$\digr(B)$ is scwaltcy. In particular we  also handle the reducible (not strongly connected) case. 
However we do not prescribe the moduli of diagonal entries, and include these moduli in the row sums instead.
This allows us, in particular, to combine the problem statement of Boros, Brualdi, Crama and 
Hoffman~\cite{BBCH} with that of
Al'pin~\cite{A}, Elsner and van den Driessche \cite{EvdD} and to generalize all 
above mentioned results removing the restriction that $B$ is irreducible.
The main results of this paper  characterize the Perron roots or the sets of eigenvalues 
of the classes of matrices under consideration.

At the end of the paper we present a new constructive proof of the 
Camion-Hoffman theorem~\cite{CH} (see also~\cite{CopH}). This
theorem characterizes regularity of a class of complex matrices with prescribed 
moduli of
their entries. The scaling result of Section~\ref{ss:perronvis} is crucial for
our new proof (which also makes use of one of the previously mentioned
 characterization results). Since we are dealing with complex rather than 
with nonnegative matrices here, the triangle inequlity (implicit in Lemma~\ref{l:camhoff})
also plays a role. 

Other proofs of the Camion-Hoffman theorem have been given by Levinger and
Varga~\cite{LV}, and Engel~\cite{E}.

\subsection{Contents of the paper}

The rest of this paper is organized as follows.
Section~\ref{ss:fnf} is a reminder of
the Frobenius normal form of nonnegative matrices. 

Section~\ref{s:vis} is devoted to a form of diagonal similarity scaling called visualization
scaling~\cite{SSB} or Fiedler-Pt\'{a}k scaling~\cite{FP} (see also~\cite{Afr}).  
Interest in this scaling has been motivated by its use in max algebra, see for 
example~\cite{But} and~\cite{BS}. 
%In Section~\ref{s:vis} take a note of 
Lemmas~\ref{l:strictvis-max} and~\ref{l:strictvis-min} can be used to 
generalize the simultaneous
scaling results of~Boros, Brualdi, Crama and Hoffman~\cite[Theorems 1.1-1.3]{BBCH} 
to include the reducible 
case. This also yields a derivation of the bounds of 
Al'pin, Elsner and van den Driessche (Theorem~\ref{t:aevdd}).
Theorem~\ref{t:perronvis} establishes the existence of 
an advanced visualization scaling, which is applied
in the proof of the Camion-Hoffman theorem.

In Section~\ref{s:nonneg} we consider the class of nonnegative matrices with prescribed graph and prescribed row sums.  
Theorem~\ref{t:mainres} characterizes the set of 
possible Perron roots of such matrices also when $B$ is reducible. This is 
one of the main results 
of this paper. The proof is based on analyzing the sunflower subgraphs of $\digr(B)$, a technique well-known in 
max algebra~\cite{HOW}.   As an immediate corollary it follows from  
Theorem~\ref{t:mainres} that for irreducible $B$
with   $\nu(B)  < \mu(B)$  
and any  $r, \nu(B) < r < \mu(B)$ there exists $A$ with $\Aux(A) = B$ such 
that $\rho(A) = r$.

In Section~\ref{ss:complex} we consider the class 
of complex matrices with prescribed graph and prescribed row sums of the moduli of their entries.
We seek a characterization of the set of nonzero eigenvalues of such matrices, 
starting with the irreducible case in Theorem~\ref{t:crhorange-irred} . 
In this case we show in particular that when $B$ has more than one cycle, the set 
of possible nonzero eigenvalues of $A$ satisfying $\Aux(A) = B$ consists either of all $s$ 
satisfying $0 < |s| < \mu(B)$ when $\nu(B) < \mu(B)$, or  $0 < |s| \leq \mu(B)$ 
if $\nu(B) =  \mu(B)$.
 Then, based on the irreducible case the full characterization in the reducible case is given in
Theorem~\ref{t:crhorange-red}. In addition to this, the occurance of a $0$ 
eigenvalue is treated in Theorem~\ref{t:zero}. 

In Section~\ref{ss:camhoff} a new proof of the Camion-Hoffman theorem~\cite{CH} is given,
based on the advanced visualization scaling 
of Section~\ref{ss:perronvis}  and the characterization result of Theorem~\ref{t:crhorange-red}.

\subsection{Frobenius normal form}

\label{ss:fnf}

Let $A$ be a square nonnegative matrix. If $A$ is irreducible (i.e., the associated digraph
is strongly connected) then according to the Perron-Frobenius theorem A has a unique (up to a multiple)
positive eigenvector corresponding to the Perron root $\rho(A)$ (which is also the greatest modulus of all
eigenvalues of $A$). If $A$ is reducible then by means of simultaneous permutations of rows and columns or, equivalently, an 
application of $P^{-1}AP$ similarity where $P$ is a permutation matrix, $A$ can be
brought to the following form:
\begin{equation*}
\begin{pmatrix}
A_1 & 0      &   0      & 0   \\
*   &  A_2   &   0      & 0   \\
*   &  *     &  \ddots & 0    \\
*   &  *     &   *     & A_m  
\end{pmatrix}.
\end{equation*}
where the square blocks $A_1,\ldots, A_m$ correspond to the maximal 
strongly connected components of the associated graph.
These diagonal blocks $A_1,\ldots, A_m$ will be further referred to as {\em classes} of $A$.
Note that 
each class $A_i$ is either a nonzero irreducible matrix, in which case it is called
{\em nontrivial}, or a zero diagonal entry (and then it is called {\em trivial}).
If some component $\digr(A_i)$ of the associated graph $\digr(A)$ 
does not have access to any other component, which means that 
there is no edge connecting one of its 
nodes to a node in another component, then this component or the corresponding class $A_i$ are called {\em final}.
Otherwise, this component or the corresponding class are called {\em transient}.

The entries denoted by $0$ are actually off-diagonal blocks of zeros of appropriate dimension,
and $*$ denote submatrices of
approriate dimensions whose zero-nonzero pattern is unimportant.

\section{Visualization scaling}

\label{s:vis}

\subsection{Visualization of auxiliary matrices}

\label{ss:visaux}

In this section we assume that $A$ is a nonnegative matrix such that 
$\digr(A)$ contains at least one cycle.
%Occasionally we will also need a matrix $\Aux'(A)$, which results from $\Aux(A)$ when
%all $0$ entries are replaced by $+\infty$ (USED ONLY ONCE, CAN BE MADE REDUNDANT)
Let us introduce some terminology related to max algebra and visualization. 

\begin{definition}
\label{def:critanti}
For a nonnegative matrix $A$, the {\rm critical graph} $\crit(A)=(N_c(A),E_c(A))$ is defined as 
the subgraph of $\digr(A)$ consisting of all nodes $N_c(A)$ and edges $E_c(A)$ on the
cycles whose geometric mean equals $\mu(A)$. These nodes and edges are also called critical. 
A node is called {\rm strictly critical} if all edges emanating from it are critical.

Similarly, by {\rm anticritical graph} we mean the subgraph of $\digr(A)$ consisting of all nodes and edges on the
cycles whose geometric mean equals $\nu(A)$ (also speaking of anticritical nodes and edges).
A node is called {\rm strictly anticritical} if all edges emanating from it are anticritical.
\end{definition}

\begin{definition}
\label{def:vis}
A positive vector $x$ is called a {\rm visualizing}, resp. {\rm strictly visualizing, vector} of $A$ 
if $a_{ij}x_j\leq \mu(A)x_i$ for all $(i,j)\in E(A)$, resp. if also $a_{ij}x_j=\mu(A)$ 
if and only if $(i,j)$ is critical. 
\end{definition}
Existence of such vector was 
proved by Engel and Schneider~\cite[Theorem 7.2]{ES}
in the irreducible case, and was extended to reducible matrices in~{\cite{SSB}.

\begin{definition}
\label{def:antivis}
A positive vector $x$ is called an {\rm antivisualizing}, resp. a {\rm strictly antivisualizing, vector} of $A$ 
if $a_{ij}x_j\geq \nu(A)x_i$ for all $(i,j)\in E(A)$, resp. if also $a_{ij}x_j=\nu(A)$ 
if and only if $(i,j)$ is anticritical. 
\end{definition}
An existence of such scaling follows from the existence of 
visualization scaling, applied to a matrix resulting from $A$ after elementwise inversion of the entries.

The following lemmas are based on the results on simultaneous scaling found in~\cite[Theorems 1.1-1.3]{BBCH}.
We make arguments of~\cite{BBCH} more precise by basing them on the existence of strictly visualizing vectors~\cite{SSB}.

\begin{lemma}[cf.~\cite{BBCH}]  
\label{l:strictvis-max}
Let $A$ be a nonnegative matrix and let $B = \Aux(A)$ with $\mu(B)\neq 0$. Let $x$ be a strictly visualizing vector 
of $B$. Then we have $Ax\leq \mu(B)x$ and, more precisely, $(Ax)_i=\mu(B)x_i$ if $i$ is a strictly critical node of
$B$ and $(Ax)_i<\mu(B)x_i$ otherwise.
\end{lemma}
\begin{proof}
Assume that $\mu(B)=1$. Then 
\beq 
\label {maxvec}
\begin{split}
&\max_j \frac{b_{ij}x_j}{x_i}\leq  1 \quad \text{for all $i$},\\
&\max_j \frac{b_{ij}x_j}{x_i}= 1 \quad \text{for all critical $i$}.
\end{split}
\eeq

If $i$ is strictly critical, we have
\begin{equation}
\forall j: (i,j)\in E(A)\ \frac{b_{ij}x_j}{x_i}=1 ,\\
\end{equation}
which implies that $x_j=x_k$ for all $j$ and $k$ such that both $(i,j)\in E(A)$ and
$(i,k)\in E(A)$. Hence we can take any $k$ with $(i,k)\in E(A)$, and obtain
\begin{equation}
\frac{\sum_j a_{ij} x_j}{x_i}=\frac{(\sum_j a_{ij}) x_k}{x_i}= \frac{b_{ik} x_k}{x_i}=1
\end{equation}

If $i$ is not strictly critical then let us denote 
\beq \label{defm}
x_k = \max_j\{x_j \colon (i,j)\in E(A)\}.
\eeq
If $i$ is not critical then
\begin{equation}
\frac{\sum_j a_{ij} x_j}{x_i}\leq\frac{(\sum_j a_{ij}) x_k}{x_i}= \frac{b_{ik} x_k}{x_i}<1
\end{equation}
If $i$ is critical (but not strictly) then
\begin{equation}
\exists l,h\ \frac{b_{il}x_l}{x_i}=1,\  \frac{b_{ih}x_h}{x_i}<1,
\end{equation}
which implies $x_l=x_k>x_h$ for these $l$ and $h$. 
In particular, note that we have $(i,k)\in E_c(B)$.
Hence
\begin{equation}
\frac{\sum_j a_{ij} x_j}{x_i}<\frac{(\sum_j a_{ij}) x_k}{x_i}= \frac{b_{ik} x_k}{x_i}=1.
\end{equation}
\end{proof}

\begin{lemma}[cf.~\cite{BBCH}]  
\label{l:strictvis-min}
Let $A$ be a nonnegative matrix and let $B = \Aux(A)$ with $\nu(B)\neq 0$. Let $x$ be a strictly antivisualizing vector 
of $B$. Then we have $Ax\geq \nu(B)x$ and, more precisely, $(Ax)_i=\nu(B)x_i$ if $i$ is a strictly anticritical node of
$B$ and $(Ax)_i>\nu(B)x_i$ otherwise.
\end{lemma}

\subsection{Bounds of Alpin, Elsner, van den Driessche}

\label{ss:aevdd}

We call a nonnegative matrix $A$ {\em truly substochastic}, if 
$\sum_j a_{ij} \leq 1$ for all $i$ and  
$\sum_j a_{ij} < 1$ for some $i$. In a similar way, $A$ is called
{\em truly superstochastic} if
$\sum_j a_{ij} \geq 1$ for all $i$ and  
$\sum_j a_{ij} > 1$ for some $i$. 

The following known result can be now obtained from Lemmas~\ref{l:strictvis-max} and~\ref{l:strictvis-min}.

\begin{theorem}[\cite{A}, {\cite{EvdD}[Theorem A]}]
\label{t:aevdd}
Let $A$ be an irreducible nonnegative matrix and let $B=\Aux(A)$.
\begin{itemize}
\item[(i)] 
If $\mu(B)=\nu(B)$, then $A$ is diagonally similar to a stochastic matrix multiplied by $\mu(B)$.  
In this case, $\rho(A) = \mu(B)=\nu(B)$.
\item[(ii)] 
If $\nu(B)<\mu(B)$, then $A$ is diagonally similar to a truly substochastic matrix multiplied by $\mu(B)$.
In this case, $\nu(B)<\rho(A)<\mu(B)$.
\end{itemize}
\end{theorem}
\begin{proof}

(i): As $B$ is irreducible and $\mu(B)=\nu(B)$, all nodes of $\digr(B)$ are strictly
critical. Taking any visualization\footnote{not necessarily strict} 
$x$ of $B$ we have $Ax=\mu(B)x$, which implies $\rho(A)=\mu(B)=\nu(B)$.
We also have that $X^{-1}AX$, with $X=\diag(x)$, is a stochastic matrix multiplied by $\rho(A)=\mu(B)=\nu(B)$

(ii): As $\mu(B)>\nu(B)$, not all nodes of $\digr(B)$ are strictly critical. 
Taking any strictly visualizing vector $x$ of $B$ we have $Ax\leq\mu(B) x$ where $(Ax)_i<\mu(B) x_i$
for some $i$.
We also have that $X^{-1}AX$ with $X=\diag(x)$, is a truly substochastic
matrix multiplied by $\mu(B)$, as claimed. As $X^{-1}AX$ is also irreducible, it 
follows that $\rho(A)=\rho(X^{-1}AX)< \mu(B)$. The inequality $\rho(A)<\mu(B)$
can be also obtained (following an argument found, for instance in~\cite{EvdD}) 
by multiplying the system $Ax\leq \mu(B)x$, where at least one of the inequalities is strict, 
from the left by a row vector $z$ such that $z A= \rho(A)z$ (which does not have $0$ components if $A$ is irreducible).   

Not all nodes of $\digr(B)$ are strictly anticritical, either.
Taking any strictly antivisualizing vector $y$ of $B=\Aux(A)$ we have $Ay\geq\nu(B) y$ where $(Ay)_i>\nu(B) y_i$
for some $i$.
We also have that $Y^{-1}AY$ with $Y=\diag(y)$, is a truly superstochastic
matrix multiplied by $\mu(B)$, as claimed. As $Y^{-1}AY$ is also irreducible, it 
follows that $\rho(A)=\rho(Y^{-1}AY)>\nu(B)$. The inequality $\rho(A)>\nu(B)$
can be also obtained 
by multiplying the system $Ay\geq \mu(B)y$, where at least one of the inequalities is strict, 
from the left by a row vector $z$ such that $z A= \rho(A)z$ (which does not have $0$ components if $A$ is irreducible).  
\end{proof}

\subsection{Sum visualization}

\label{ss:perronvis}

%It turns out to be possible to visualize an irreducible matrix so that it is also superstochastic.

\begin{definition}

\label{def:asumvis}
For $A\in\Rpnn$ and $a > 0$, a vector $x\in\Rpn$ is called an $a$-sum visualizing vector
of $A$, if the entries of $C = X^{-1}AX$ with $X = \diag(x)$ satisfy $c_{ij}\leq a$ for
all $i,j$ and $\sum_j  c_{ij}\geq  a$ for all $i$. In this case C is called an $a$-sum 
visualization of $A$.
\end{definition}

Recall that we have $\mu(A)\leq\rho(A)$ for any nonnegative matrix. Indeed, since for
any positive $x$ and any cycle $(i_1,\ldots i_{\ell})$ we have that 
$$\left(a_{i_1i_2}\frac{x_{i_2}}{x_{i_1}}\cdot a_{i_2i_3}\frac{x_{i_3}}{x_{i_2}}
\cdot\ldots\cdot a_{i_{\ell}i_1}\frac{x_{i_{\ell}}}{x_{i_1}}\right)^{1/{\ell}}
\leq \left(\prod_{k\in\{i_1,\ldots,i_{\ell}\}} \sum_j 
a_{kj}\frac{x_j}{x_k}\right)^{1/{\ell}},$$
it follows by taking $x$ satisfying $Ax=\rho(A)x$, that 
$\mu(A)\leq\rho(A)$. 

\begin{theorem} 
\label{t:perronvis}
Let $A\in\Rpnn$ be irreducible, and define $\alpha(A)$ 
as the set of positive numbers $a$ for which
an $a$-sum visualization of $A$ exists. Then $\alpha(A) = [\mu(A), \rho(A)]$.
\end{theorem}

\begin{proof}
1. $\alpha(A)\subseteq [\mu(A),\rho(A)]$: 

\medskip\noindent Let $a\in\alpha(A)$ and let $C=X^{-1}AX$ (for some 
diagonal $X$) be such that $c_{ij}\leq a$ for all $i,j$ and
$\sum_j c_{ij}\geq a$ for all $i$.  Then $\mu(C)\leq a$ and
$\rho(C)\geq a$, and as $\mu(A)=\mu(C)$ and $\rho(A)=\rho(C)$
we obtain that $a\in[\mu(A),\rho(A)]$.

2. $[\mu(A),\rho(A)]\subseteq\alpha(A)$:

\medskip\noindent Let $\mu(A)\leq a\leq\rho(A)$.
We can assume without loss of generality (dividing $A$ by $a$
if necessary) that $a=1$ and $\mu(A)\leq 1\leq \rho(A)$.

As $\mu(A)\leq 1$, 
there exists a non singular diagonal matrix $X$ such that all entries $g_{ij}$ of 
$G:=X^{-1} A X$ satisfy $0\leq g_{ij}\leq 1$. Since $G$ is diagonally similar to $A$, 
$\rho(A)$ is also the spectral radius of $G$ and hence the exists a vector $z$ whose entries $z_i$ satisfy 
$1=\max_i z_i$ and $\sum_j g_{ij}\frac{z_j}{z_i} \geq 1$ for all $i$ .

We will now construct an entrywise nonincreasing sequence of vectors $\{y^{(s)}\}_{s\geq 0}$
bounded from below by $z$. Such a 
sequence obviously converges, and as we will argue, the limit denoted by $y$ satisfies $g_{ij}\frac{y_j}{y_i} \leq 1$
for all $i, j$, and $\sum_j g_{ij}\frac{y_j}{y_i} \geq 1$ for all $i$ (and,  obviously, $y\geq z$).

Let us define a continuous mapping $f\colon (\Rp\backslash\{0\})^n \to (\Rp\backslash\{0\})^n$, by its components
\begin{equation}
\label{def:fi}
f_i(x)=\min(x_i,\sum_j g_{ij}x_j),\quad i=1,\ldots,n.
\end{equation}

% which entry by entry  is non increasing and is bounded below by the entries of Z. 

Now let $y^{(0)}=(1,1\ldots 1)$ and consider a sequence $\{y^{(s)}\}_{s\geq 0}$ 
defined by $y^{(s+1)}:= f(y^{(s)})$ (that is, the orbit of $y^{(0)}$ under $f$).

Observe that $y^{(s+1)}\leq y^{(s)}$, as $f(x)\leq x$ for all $x\in (\Rp\backslash\{0\})^n$. 
%Furthermore, $y^{(s+1)}_k<y^{(s)}_k$ for all indices $k\in I(y^{(s)})$ (also since
%$y^{(s+1)}=f(y^{(s)})$.

It follows by induction that $y^{(s)}\geq z$ for all $s$. The case $s=0$ is the basis of induction (since 
$z_i\leq 1$ for all $i$). We have to show that 
$y^{(s+1)}\geq z$ knowing that $y^{(s)}\geq z$. It amounts to verify that
$y^{(s+1)}_k\geq z_k$ for the indices $k$ where $y^{(s+1)}_k<y^{(s)}_k$. For such indices 
we have 
$$
y_k^{(s+1)}=\sum_j g_{kj} y_j^{(s)}\geq \sum_j g_{kj} z_j^{(s)}\geq z_k.
$$

As the sequence $\{y^{(s)}\}_{s\geq 0}$ is nonincreasing and bounded from below, it has a limit
which we denote by $y$. As $f$ is continuous, this limit satisfies $f(y)=y$, which by the definition of
$f$  implies that $\sum_j g_{ij}\frac{y_j}{y_i}\geq 1$ for all $i$.

We now show by induction that  $g_{ij}\frac{y^{(s)}_j}{y^{(s)}_i}\leq 1$, for all $i\neq j$ and $s$. 
Denote by $I_s$ the set of indices $i$ where $\sum_j g_{ij} y_j^{(s)}< y_i^{(s)}$ . Thus
$y^{(s+1)}_i= \sum_j g_{ij} y_j^{(s)}$ and $y_i^{(s+1)}<y_i^{(s)}$ for $i\in I_s$, while
$y^{(s+1)}_i=y^{(s)}_i$ for $i\notin I_s$.

Observe that $s=0$ is the basis
of induction, so we assume that the claim holds for $s$ and we have to prove it for $s+1$. 
For $i,j\notin I_s$ the inequality $g_{ij}\frac{y^{(s+1)}_j}{y^{(s+1)}_i}\leq 1$ holds trivially.
If $i\in I_s$ then
$$g_{ij}\frac{y^{(s+1)}_j}{y^{(s+1)}_i}\leq g_{ij}\frac{y^{(s)}_j}{y^{(s+1)}_i}=
g_{ij} y^{(s)}_j(\sum_k g_{ik} y^{(s)}_k)^{-1}\leq 1$$
(where the last inequality follows since $g_{ij} y^{(s)}_j$ is just one of the nonnegative
terms of the sum in the denominator).

Finally if $i\notin I_s$ and $j\in I_s$: then we have 
$g_{ij}\frac{y^{(s+1)}_j}{y^{(s+1)}_i}< g_{ij}\frac{y^{(s)}_j}{y^{(s+1)}_i}= 
g_{ik}\frac{y^{(s)}_j}{y^{(s)}_i}\leq 1$.

Thus the inequalities $g_{ij}\frac{y^{(s)}_j}{y^{(s)}_i}\leq 1$ hold for all $i\neq j$ and $s$, and this implies that
for the limit point $y$,  all the inequalities  $g_{ij}\frac{y_j}{y_i}\leq 1$ hold as well. 
The case $i=j$ is trivial since the inequality $g_{ii} \frac{y_i}{y_i}=g_{ii}\leq 1$ holds for all $i$.

%Since each coordinate of  $y^{(s)}$ is non increasing as s increases and is bounded below by the corresponding coordinate of Z it converges. Denote the limit by y.
%We note that we have $0\leq g_{ij}\frac{y^{(s)}_j}{y^{(s)})_i}\leq 1$ for all i,j  and for all $s$ since this inequality holds for $s=0$ and then follows by induction from step 2. Since the sequence $y^{(0)},\ldots,y^{(s)},\dots$ converges we have that given $\epsilon>0$ there exists $s_*$ such that for $s>s_*$ $\max_i | \frac{y^{s+1}_i}{y^{(s)}_i}-1|<\epsilon$. Thus  for $s>s_*$ if $\Sigma_j g_{kj}\frac{y^{(s)}_j}{y^{(s)_k}}<1$ then $| \Sigma_j g_{kj}\frac{y^{(s)}_j}{y^{(s)}_k}-1|<\epsilon$ and hence $\sum g_{ij}\frac{y_j}{y_i}\geq 1$ for all i. 
 %\newline

Let $D$ be the diagonal matrix with 
$d_{ii}=y_ix_i$ for all i. For the entries $c_{ij}$ of $C=D^{-1}A D$ we have 
$c_{ij} \leq 1$ for all $i, j$ and $\sum_j c_{ij} \geq 1$ for all $i$ so the theorem is proved.  
\end{proof} 

Denote by $A^{[-1]}=(a_{ij}^{[-1]})$ the Hadamard inverse of $A\in\Rpnn$:
$$
a^{[-1]}_{ij}=
\begin{cases}
\frac{1}{a_{ij}}, & \text{if $a_{ij}>0$},\\
0, & \text{if $a_{ij}=0$}.
\end{cases}
$$
Observe that $\mu(A^{[-1]})=(\nu(A))^{-1}$ (however, there is no such inversion for the
Perron root), and let us formulate the following corollary of Theorem~\ref{t:perronvis}.

\begin{corollary} 
\label{c:perronvis}
Let  $A\in\R_{+}^{n\times n}$ be irreducible. The following are equivalent:
\begin{itemize}
\item[(i)] $\frac{1}{a}\in[\frac{1}{\nu(A)},\rho(A^{[-1]})]$;
%\min_{x>0}\max_i \sum_{a_{ij}\neq 0} \frac{1}{a_{ij}}\frac{x_j}{x_i}]$
\item[(ii)] $\exists x>0$ such that for $C=X^{-1} AX$  with $X=diag(x)$ we 
have that $c_{ij}\geq a$ for all $i,j$. and  $\sum_j \frac{a}{c_{ij} }\geq 1$
for all $i$.
\end{itemize}
\end{corollary}
\begin{proof}
The corollary follows by elementwise inversion of 
the nonzero entries and applying Theorem~\ref{t:perronvis}. 
\end{proof}

\section{Nonnegative reducible matrices}
\label{s:nonneg}

Here we characterize Perron roots of nonnegative matrices with prescribed 
row sums and prescribed graph.  Section~\ref{ss:sunflo} is devoted to sunflower graphs,
which will be used in the proof of the main result.  Section~\ref{ss:nonneg}
contains the main result and example.

\subsection{Sunflowers}
\label{ss:sunflo}

We introduce the following definition,  
inspired by description of the Howard algorithm in~\cite{Coc+} and \cite[Chapter 6]{HOW}. 

\begin{definition}
Let $\digr$ be a weighted graph. A subgraph $\tilde{\digr}$ of $\digr$ is called a {\rm sunflower subgraph} of $\digr$ if
the following conditions hold:
\begin{itemize}
\item[{\rm (i)}] If a node in $\digr$ has an outgoing edge then it has a unique 
outgoing edge in $\tilde{\digr}$; 
\item[{\rm (ii)}] Every edge in $\tilde{\digr}$ has the same weight as the corresponding edge
in $\digr$.
\end{itemize} 
\end{definition}

It is easy to see (\cite{HOW}) 
that such a digraph can be decomposed into several isolated components,
each of them either acyclic or consisting of a unique cycle and some walks leading to it. 
A sunflower subgraph $\tilde{\digr}$ of $\digr$ is called a simple $\gamma$-sunflower subgraph of $\digr$,
if $\gamma$ is the unique cycle of $\tilde{\digr}$.
The set of all sunflower subgraphs of the weighted digraph 
$\digr(B)$, with full node set $1,\ldots, n,$ will be denoted
by $\cS(B)$.

Denoting by $\mu(\digr)$ the maximal cycle mean of a subgraph $\digr\subseteq\digr(B)$, we introduce the
following parameters:
\begin{equation}
\label{e:mama}
M(B):=\max\limits_{\digr\in\cS(B)} \mu(\digr),\quad m(B):=\min\limits_{\digr\in\cS(B)} \mu(\digr).
\end{equation}

%It can be shown that $M(A)=\mu(A)$ for any nonnegative matrix $A$, and also that $m(A)=\nu(A)$ if
%$A$ is irreducible. However, in general we only have $\nu(A)\leq m(A)\leq \mu(A)=M(A)$. 
%Let us recall the following graph-theoretic fact.

%The following results are standard and their proofs will be omitted.

\begin{lemma}
\label{l:sunfl-exist} 
Let $\digr$ be a strongly connected graph. 
Then, for any cycle $\gamma$ of $\digr$ there exists a simple 
$\gamma$-sunflower subgraph of $\digr$. 
\end{lemma}

\begin{proof} Let $\{1,\ldots,n\}$ be the nodes of $\digr$.
Suppose that $\{1,\ldots,k\}$ are the nodes in $\gamma$, and 
$k+1,...,n$ are the rest of the nodes. 
 
Observe first that we can construct a simple $\gamma$-sunflower on nodes $1,\ldots,k$:
this is just the cycle $\gamma$ itself.

The proof is by contradiction. Assume that a simple $\gamma$-sunflower $\Tilde{\digr}$ 
can be constructed for a subgraph 
induced by the set of nodes $M$, which 
contains the nodes $1,\ldots,k$ and is a {\bf proper} subset of 
$\{1,\ldots, n\}$, and that $M$ is a maximal such set. 
However, since $\digr$ is connected, 
there is a walk $W$ from $\{1,\ldots,n\}\backslash M$ to $M$, 
and we can pick the last edge of that walk and its last node before it enters $M$. 
Adding that node and that edge to $\Tilde{\digr}$ we increase it 
while it remains a simple $\gamma$-sunflower (of a subgraph induced by a larger node set).
The contradiction shows that we can construct a simple $\gamma$-sunflower of $\digr$.
\end{proof}

Let us also recall the following.

\begin{lemma}
\label{l:sunfl-rho}
Let $A$ be a nonnegative square matrix  
such that the digraph associated with $A$ is a sunflower graph. Then
$\rho(A)=\mu(A)$.
\end{lemma}

\begin{proof} Clearly, the cycles of $\digr(A)$ are exactly the 
nontrivial classes of the Frobenius Normal Form. Hence it suffices to observe
that $\rho(A)=\mu(A)$ if $\digr(A)$ is a Hamiltonian cycle $\gamma$. Indeed, we can
set $x_i=1$ for any $i\in\gamma$ and then
calculating all the rest of coordinates from the equalities $a_{ij}x_j=\mu(A) x_i$ for $a_{ij}\neq 0$. 
This computation does not lead to contradiction, since $\mu(A)$ is the cycle mean of $\gamma$.
\end{proof}

The following proposition expresses $m(B)$ and $M(B)$
in terms associated with the Frobenius normal form.

\begin{proposition}
\label{p:mama2}
Let $B$ be a nonnegative matrix. Then
\begin{equation}
\label{e:mama2}
M(B)=\mu(B),\quad m(B)=\max\limits_{N_i \text{is final}} \nu(B_i). 
\end{equation}
\end{proposition}
\begin{proof}

\medskip\noindent $M(B)$: It is obvious from~\eqref{e:mama} that $M(B)\leq\mu(B)$. The reverse inequality
$M(B)\geq \mu(B)$ follows since we can take a cycle $\alpha$ of $\digr(B)$ whose cycle mean equals to $\mu(B)$ and
construct a sunflower subgraph of $\digr(B)$  
that contains $\alpha$ as one of its cycles.

\medskip\noindent $m(B)$: It is obvious from~\eqref{e:mama} that $m(B)\geq\max\limits_{N_i \text{is final}} \nu(B_i)$,
since any sunflower subgraph of $B$ contains a cycle in every nontrivial final class. So we show that
$m(B)\leq\max\limits_{N_i \text{is final}} \nu(B_i)$.  For this, in each submatrix $B_i$ corresponding to
a final class we take a cycle $\alpha_i$ whose mean value is $\nu(B_i)$ and using Lemma~\ref{l:sunfl-exist} build a simple
$\alpha_i$ sunflower of the strongly connected component associated with $B_i$. Unite all these sunflowers. If $B_i$ is not final then
it has access to another class from some node $k_i$. In this case build a spanning tree on the nodes of $B_i$,
directed to $k_i$, and for $k_i$ choose an edge going to another class.  
Finally, for each trivial node of $B_i$ we choose an arbitrary outgoing edge if it exists. Adjoin these 
spanning trees and outgoing
edges to the above union of simple sunflowers. This leads to a sunflower subgraph $\digr$ of $\digr(B)$, for which we have
$\mu(\digr)=\max\limits_{N_i \text{is final}} \nu(B_i)$, hence $m(B)\leq\max\limits_{N_i \text{is final}} \nu(B_i)$ 
and the required equality follows. 
\end{proof}

\begin{remark}
Observe that $m(B)=0$ if and only if all final classes of $B$ are trivial.
\end{remark}

A sunflower subgraph which has cycles only in the final classes of $\digr(B)$ will be called {\em thin}.
In the proof of Proposition~\ref{p:mama2} we actually established the following result.

\begin{lemma}
Let $\digr$ be a graph where each node has an outgoing edge and let $\digr_i$ for 
$i=1,\ldots, q$ be the nontrivial final components of $\digr$.\\
For each collection of cycles $\alpha_i\in\digr_i$ for $i=1,\ldots, q$, there is a (thin) sunflower subgraph of
$\digr$ whose cycles are $\alpha_1,\ldots, \alpha_q$.\\
If all final components of $\digr$ are trivial then 
there exists an acyclic sunflower subgraph of $\digr$ (i.e., a directed forest).
\end{lemma}

\subsection{Range of the Perron root}
\label{ss:nonneg}

For a row uniform nonnegative matrix $B$, denote 
\begin{equation}
\label{e:etadef}
\eta(B):=\{\rho(A)\colon A\in\Rpnn,\ \Aux(A)=B\}.
\end{equation}

We are going to extend Theorem~\ref{t:aevdd} to include the reducible case and describe
$\eta(B)$ for a general row uniform nonnegative matrix $B$.

\begin{theorem}
\label{t:mainres}
Let $B$ be a nonnegative row uniform matrix.
\begin{itemize}
\item[(i)] $\eta(B)\subseteq [m(B),M(B)]$.
\item[(ii)] $M(B)\in\eta(B)$ if and only if there is at least one final class $B_i$ with
$\mu(B_i)=\nu(B_i)=M(B)$.
\item[(iii)] If $m(B)>0$ then $m(B)\in\eta(B)$ if and only if $\mu(B_i)=\nu(B_i)=m(B)$ for all final 
(nontrivial) classes $N_i$
attaining the maximum in~\eqref{e:mama2}. If $m(B)=0$ then $m(B)\in\eta(B)$ of and only if $\digr(B)$ is acyclic,
in which case $\eta(B)=\{0\}$.
\item[(iv)] If $M(B)=m(B)$ then $\eta(B)=\{m(B)\}$.
\item[(v)] If $M(B)>m(B)$ 
then $(m(B),M(B))\subseteq\eta(B)$.
\end{itemize}
\end{theorem}
\begin{proof} 
Throughout the proof, let $A$ be such that $\Aux(A)=B$. Let $A_i$ and $B_i$ for $i=1,\ldots,m$ be the classes of the Frobenius normal form of $A$ and $B$
respectively, and let $N_i$ be the corresponding node sets (or classes).

\medskip\noindent (i): We have to show that $\rho(A)\in[m(B),M(B)]$. Note that for any class $A_i$ of $A$ we have
$\Aux(A_i)\leq B_i$, and Theorem~\ref{t:aevdd} implies that $\rho(A_i)\leq\mu(B_i)$, 
but we do not have $\rho(A_i)\geq\nu(B_i)$ in general. However, 
$\Aux(A_i)=B_i$ holds for a final class, and hence $\nu(B_i)\leq\rho(A_i)\leq\mu(B_i)$ for 
any final class.

With above considerations, the inequality $\rho(A)\leq M(B)$ follows since $M(B)=\mu(B)$ and $\rho(A_i)\leq\mu(B_i)$
for all classes. To show that $\rho(A)\geq m(B)$ we first define matrix $\Tilde{A}$ formed from $A$ by
zeroing out all the entries except for the entries in final classes, and we similarly define $\Tilde{B}=\Aux(\Tilde{A})$. 
Then we have $\rho(A)\geq\rho(\Tilde{A})$ by 
monotonicity of the spectral radius.  Since $m(B)=\max\limits_{N_i \text{is final}} \nu(B_i)$ by~\eqref{e:mama2}
and $\rho(A_i)=\rho(\Tilde{A}_i)\geq\nu(\Tilde{B}_i)=\nu(B_i)$ for each final class we obtain
that $\rho(A)\geq\rho(\Tilde{A})\geq m(B)$,
hence the claim.

\medskip\noindent (ii): When $\digr(B)$ is acyclic the proof of (ii)
is trivial. If $\digr(B)$ is not
acyclic then $M(B)=\mu(B)>0$ and if $\rho(A_i)=\mu(B)$ then $A_i$ must be nontrivial.
 We first argue that $\rho(A_i)=M(B)$ is impossible if $A_i$ has access to other classes.
Indeed, if there is such access then we only have 
$\Aux(A_i)\leq B_i$ with strict inequalities
in some rows. This implies that we can find $A'_i$ such that $\Aux(A'_i)=B_i$ 
and $A'_i\geq A_i$,
with strict inequalities in the same rows. But then we have $\rho(A_i)<\rho(A'_i)\leq\mu(B_i)$
so $\rho(A_i)=\mu(B)$ is impossible. Thus $\rho(A_i)=\mu(B)$ can be attained only in a final class, which happens if
and only if $\rho(A_i)=\mu(B_i)$, and by Theorem~\ref{t:aevdd}, if and only if 
$\mu(B_i)=\nu(B_i)=\mu(B)$ for one such class.

\medskip\noindent (iii): In the case when $m(B)=0$ but $B$ has at least one nontrivial class 
$B_i$, we have $\rho(A_i)>0$ and hence $\rho(A)>0$ for any $A$ such that $\Aux(A)=B$.  Therefore in this case
$m(B)=0\in\eta(B)$ if and only if all classes of $B$ are trivial (i.e., $\digr(B)$ is acyclic). 

If $m(B)>0$, we first show that the given condition is necessary: if $\nu(B_i)<\mu(B_i)$ for at least one of these final classes then
we have $\nu(B_i)<\rho(A_i)<\mu(B_i)$ by Theorem~\ref{t:aevdd}, hence $m(B)<\rho(A_i)\leq\rho(A)$, so $m(B)=\rho(A)$
does not hold. Following the proof of Proposition~\ref{p:mama2},
we can construct a thin sunflower subgraph of $\digr(B)$ with the cycles attaining
$\nu(B_i)$ in all final classes. Denote
the matrix associated with this subgraph by $C$ and the submatrices extracted from the node sets $N_i$ by $C_i$.
For each submatrix $C_i$ with $N_i$ not final, we have $\rho(C_i)=0$. By the continuity of Perron root
we can find a small enough $\epsilon$ such that 
$\rho((1-\epsilon)C_i+\epsilon A_i)$ is smaller than $m(B)$ for all classes that are not final and for
all classes that are final but have $\nu(B_i)<m(B)$. This is while $\Aux((1-\epsilon)C+\epsilon A)=B$ and, by 
Theorem~\ref{t:aevdd}, $\rho((1-\epsilon)C_i+\epsilon A_i)=m(B)$ for all classes where the maximum in~\eqref{e:mama2} is attained.
This implies that $\rho((1-\epsilon)C+\epsilon A)=m(B)$, hence the claim. 

\medskip\noindent (iv): By part (i), $\rho(A)$ can be only equal to $m(B)=M(B)$. However, the set of $A$ such that
$\Aux(A)=B$ is nonempty for any row uniform $B$, hence the claim. 
%Alternatively, $m(B)=M(B)$
%and~\eqref{e:mama2} imply that $\mu(B)$ is attained at some final classes, and we have $\nu(B_i)=\mu(B_i)$
%for each class where $\nu(B_i)=m(B)=M(B)$, so that $m(B)=M(B)\in\eta(B)$ by part (ii) or part (iii).

\medskip\noindent (v): Let us first observe that by definition of $m(B)$ and $M(B)$~\eqref{e:mama}, there exist 
matrices $\underline{A}$ and $\overline{A}$ whose associated graphs are the sunflower subgraphs of
$\digr(B)$ attaining the maximum and the minimum value of $\mu(\digr)$ over all possible sunflower subgraphs
of $\digr(B)$.
%In particular we have $\Aux(\underline{A})\preceq B$ and $\Aux(\overline{A})\preceq B$. 
By Lemma~\ref{l:sunfl-rho} we have that $\rho(\underline{A})=m(B)$ and $\rho(\overline{A})=M(B)$.

Now we argue that there exists $A_0$ with $\Aux(A_0)=B$ and $\rho(A_0)$
arbitrarily close to $m(B)=\rho(\underline{A})$. Indeed, let $D$ be any matrix with $\Aux(D)=B$, and consider
the family of matrices $C_{\epsilon}=(1-\epsilon)\underline{A}+\epsilon D$ for $\epsilon>0$.
Then (for any $\epsilon>0$) we have $\Aux(C_{\epsilon})=B$ and since
$\rho(C_{\epsilon})$ is a continuous function of $\epsilon$ it
follows that $\lim_{\epsilon\to 0} \rho(C_{\epsilon})=\rho(\underline{A})$.
Similarly, there exists 
$A_1$ with $\Aux(A_1)=B$ and $\rho(A_1)$
arbitrarily close to $M(B)=\rho(\overline{A})$. 

Thus for each $\epsilon$ we have some $A_0$ and $A_1$ with 
$\Aux(A_0)=\Aux(A_1)=B$ and $\rho(A_0)<m(B)+\epsilon$
and $\rho(A_1)>M(B)-\epsilon$. For $\lambda$, where $0<\lambda<1$, let
$A_{\lambda}:=\lambda A_1+ (1-\lambda)A_0$ interpolate between $A_0$ and $A_1$.
Since $\Aux(A_{\lambda})=B$ for each $\lambda$ %we have $m(B)\leq\rho(A_{\lambda}<M(B)$ for $0\leq\lambda\leq 1$.
and $\rho(A_{\lambda})$ is continuous in $\lambda$, the claim follows.
\end{proof}

As an immediate corollary we obtain the following result in the irreducible case.

\begin{corollary}
\label{c:mainres}
Let $B$ be an irreducible nonnegative row uniform matrix. Then
\begin{itemize}
\item[(i)] 
If $\nu(B) < \mu(B)$ then $\eta(B) = (\nu(B), \mu(B))$ .
\item[(ii)] 
If $\nu(B) = \mu(B)$ then $\{\nu(B)\} = \eta(B) =  \{\mu(B)\}$.
\end{itemize}
\end{corollary}

{\bf Example.}
%Recall that A matrix $B$ is row uniform if there exists $A\in\Rpnn$
%such that $\Aux(A) =B$. 
Given an irreducible row uniform matrix $B\in\Rpnn$
and a constant $\rho\in(\nu(B);\mu(B)) = (m(B);M(B))$, we describe a method for constructing a
matrix $A$ such that $\Aux(A) = B$ and $\rho(A) =\rho$. Take two simple $\gamma$ -
sunflowers: one where $\gamma$  has cycle mean equal to $\mu(B)$, and the other
where $\gamma$ has cycle mean equal to $\nu(B)$.
Denote by $A_1$ the matrix associated with the first sunflower, and by $A_2$ the matrix associated
with the second sunflower. We have $\rho(A_1)=\mu(B)$ 
and $\rho(A_2)=\nu(B)$. For the convex combinations of these matrices,
we have that $\rho(A_{\lambda})$, where $A_{\lambda}:=(1-\lambda) A_1+\lambda A_2$ 
and $0\leq\lambda\leq 1$, will assume all values between
$\nu(B)$ and $\mu(B)$. This follows from the continuity of spectral radius 
as a function of $\lambda$(as in the more general construction above).
The value of $\lambda$ for which $\rho(A_{\lambda})=\rho$, can be found from the system
$A(\lambda)x=\rho x$, which has $n+1$ variables 
($n$ components of $x$ and the parameter $\lambda$). 
However, since $x$ can be multiplied by any scalar,
one of the coordinates of $x$ can be chosen equal to $1$. 
Then, for at least one of such choices, the existence of solution is guaranteed.

For example, consider

\begin{equation}
B=
\begin{pmatrix}
0 & 8 & 8 & 0 & 8\\
2 & 0 & 0 & 2 & 2\\
2 & 0 & 0 & 0 & 0\\
0 & 3 & 0 & 3 & 3\\
0 & 3 & 0 & 0 & 0
\end{pmatrix}
\end{equation}

We see that the cycle $(1,2)$ is critical, with the cycle mean $\mu(B)=4$, and the cycle $(2,5)$ is anticritical
with the cycle mean $\nu(B)=\sqrt{6}$. For the matrices $A_1$ and $A_2$ assiciated with the corresponding
sunflower graphs, we can take
\begin{equation}
A_1=
\begin{pmatrix}
0 & 8 & 0 & 0 & 0\\
2 & 0 & 0 & 0 & 0\\
2 & 0 & 0 & 0 & 0\\
0 & 3 & 0 & 0 & 0\\
0 & 3 & 0 & 0 & 0
\end{pmatrix},\quad
A_2=
\begin{pmatrix}
0 & 8 & 0 & 0 & 0\\
0 & 0 & 0 & 0 & 2\\
2 & 0 & 0 & 0 & 0\\
0 & 3 & 0 & 0 & 0\\
0 & 3 & 0 & 0 & 0
\end{pmatrix}
\end{equation}

Equation $A_{\lambda}x=\rho x$, where we put $x_1=1$, can be written as
\begin{equation}
\label{e:example}
\begin{pmatrix}
0 & 8 & 0 & 0 & 0\\
2-y & 0 & 0 & 0 & y\\
2 & 0 & 0 & 0 & 0\\
0 & 3 & 0 & 0 & 0\\
0 & 3 & 0 & 0 & 0
\end{pmatrix}
\begin{pmatrix}
1\\
x_2\\
x_3\\
x_4\\
x_5
\end{pmatrix}
=\rho
\begin{pmatrix}
1\\
x_2\\
x_3\\
x_4\\
x_5
\end{pmatrix},
\end{equation}
where $y\in[0,2]$ (so that $y=2\lambda$).  
Observe that $A_{\lambda}$ is irreducible, so the existence of a solution with $x_1=1$ 
(as well as a solution with
any other component set to $1$) is guaranteed.

System~\eqref{e:example} can be solved explicitly. Indeed, 
from the first equation of this system we have $8x_2=\rho$
so $x_2=\rho/8$, from the third equation we have $2= \rho x_3$ so 
$x_3=2/\rho$, from the fourth and the fifth equation we have
$3x_2=3\rho/8=\rho x_4=\rho x_5$ so $x_4=3/8=x_5$. Using 
the second equation of the system, we obtain $2-y+(3/8)y=\rho x_2=(\rho^2)/8$.
Thus $y=\frac{16-\rho^2}{5}$.

\section{Complex matrices}

In Section~\ref{ss:complex} we characterize the set of eigenvalues of complex matrices with prescribed graph
and prescribed row sums of the moduli of their entries. 

In Section~\ref{ss:camhoff}, a new proof of the Camion-Hoffman theorem is presented.

\subsection{Complex matrices with prescribed row sums of moduli}
\label{ss:complex}

\begin{definition}
\label{def:sigma}
For $B$ a row uniform nonnegative matrix,  let $\sigma(B)$ denote the set 
\newline
 $\sigma(B)=\left\{\lambda \colon\exists A\in\C^{n\times n},\ \Aux(|A|) = B,\ \det(A-\lambda I)=0 \right\}$.
\end{definition}
Here $|A|$ denotes the matrix whose entries are the moduli of (complex) entries of $A$.

%We shall prove the following result.

We first consider the conditions when $0\in\sigma(B)$. In what follows, the imaginary number ``$i$''
is denoted by $\Im$. By a {\em generalized diagonal product} of $B$ we mean a product of the form
$\prod_{i=1}^n b_{i\sigma(i)}$ where $\sigma$ is an arbitrary permutation of $\{1,\ldots,n\}$.

\begin{theorem}
\label{t:zero}
Let $B$ be a row uniform nonnegative matrix. Then the following are equivalent: 
\begin{enumerate}
	\item[{\rm (i)}] $0\in\sigma(B)$.
	\item[{\rm (ii)}] The number of generalized nonzero diagonal products of $B$ is not $1$. 
\end{enumerate}
\end{theorem}
\begin{proof}
Suppose the number of generalized nonzero diagonal products of $B$ is one. 
Let $A$ be such that $\Aux(A)=B$. 
The determinant of $A$ equals the signed sum of the nonzero 
generalized
diagonal products of $A$. Since all but one of the generalized diagonal products of $A$ are zero, we have $\det(A)\neq 0$. 
Thus $0\notin\sigma(B)$.  

Suppose that $B$ has no generalized nonzero diagonal products and 
let $A$ be such that $\Aux(A)=B$. %The determinant of A equals the signed sum of the nonzero generalized diagonal products of A. 
Since all the generalized diagonal products of $A$ are zero, we have $\det(A)= 0$ and $0\in\sigma(B)$.    

Suppose that $B$ has two or more non zero generalized diagonal products.  
Let us permute the columns of $B$ in order to put one of the generalized diagonal products
on the (main) diagonal. In other words, consider $BP$ where $P$ is a permutation matrix and all diagonal entries of
$BP$ are nonzero.
We have $\Aux(A)=B$ if and only if $\Aux(AP)=BP$, and $\det(A)=\det(AP)$, therefore $0\in\sigma(B)$ if and only if
$0\in\sigma(BP)$. As $BP$ has at least one nonzero diagonal product different from the main diagonal, 
the Frobenius normal form of $BP$ has a nontrivial diagonal block of dimension greater than $1$. 

Denote the index set of that block by $M$, and let us take any row uniform nonnegative matrix $D=(d_{kl})$ such that 
$\Aux(D)=BP$. For each $k\in M$, denote by $n_k$ the number of outgoing edges of the $k$th node in $M$ in the
associated digraph of $BP$
that go to the nodes in $M$. As the block is irreducible and has all diagonal entries nonzero, we have $n_k>1$.  
Let $t_k$ be a bijection between the outgoing edges of $k$ and $\{1,2,\ldots n_k\}$, and define
matrix $C=(c_{kl})$ by  

\begin{equation}
c_{kl}=
\begin{cases}
  d_{kl} \exp(\Im \frac{t_k(l)2\pi}{n_k}), & \text{if $k,l\in M$ and $d_{kl}\neq 0$}\\
  d_{kl}, & \text{otherwise}.
\end{cases}
\end{equation}

Then $\Aux(|C|)=BP$. In addition $C_{MM}v=0$  where $C_{MM}$ is the principal submatrix of $C$ 
extracted from rows and columns with indices in $M$, and $v$ is the vector with all components 
equal to $1$. This implies that $\det(C)=\det(C_{MM})=0$ so $0\in\sigma(BP)$ and $0\in\sigma(B)$.

%=(v_1,\ldots,v_n)$ with $v_i=0\ for\ i\notin S_2, v_i=1\ for\ i\in S_2$. Thus $\det C[nodes\ of\ G_q]=0$. Since $\det C=\prod_{i\ is\ final} \det C[nodes\ of\ G_i]$  and the result is established. 
% \eproof
%}\fi

\end{proof}

We now describe $\sigma(B)\backslash\{0\}$ starting from the
irreducible case.

\begin{definition}
An irreducible matrix $B$ is called unicyclic if $\digr(B)$ consists of a 
single Hamiltonian cycle, and multicyclic otherwise. 
\end{definition}

\begin{theorem}\label{t:crhorange-irred}
Let $B$ be a row uniform nonnegative irreducible matrix. 
\begin{itemize}

\item[{\rm (i)}] If $B$ is unicyclic then $\sigma(B)=\{s\colon |s|=\mu(B)\}$;  
\item[{\rm (ii)}] If $B$ is multicyclic and $\nu(B)<\mu(B)$ then\\
$\sigma(B)\setminus\{0\}=\{ s\colon 0<|s|<\mu(B)\}$;
\item[{\rm (iii)}] If $B$ is multicyclic and $\nu(B)=\mu(B)$ then\\
$\sigma(B)\setminus\{0\}=\{ s\colon 0<|s|\leq\mu(B)\}$.
\end{itemize}
\end{theorem}
\begin{proof}
(i): In this case, all complex matrices $A$ satisfying $\Aux(|A|)=B$ are formed 
by multiplying 
the entries of $B$ (that is, the entries of its only cycle) by some complex numbers of modulus $1$.
The claim follows.

(ii), (iii):
We first show that $\sigma(B)$ is contained in the above mentioned intervals. For that we first recall
a known result of Frobenius (see e.g.~\cite{BP}, p.31, Theorem 2.14) that for any square complex matrix $A$, we have
\begin{equation}
\label{e:Frob}
\rho(A):=\max\{|\lambda|> 0 \colon \det(A-\lambda I)=0\}\leq\rho(|A|).
\end{equation}
As $\Aux(|A|)=B$, Theorem~\ref{t:mainres} implies that $\rho(|A|)\leq \mu(B)$ if $\mu(B)\in\eta(B)$
and $\rho(|A|)<\mu(B)$ if $\mu(B)\notin\eta(B)$. Combining these inequalities with~\eqref{e:Frob}, we
have the desired inclusion.

We are left to show that each number in the intervals can be realized as an eigenvalue of a 
complex matrix $A$ with $\Aux(|A|)=B$.
Select $\lambda\in(0,\mu(B))$ if $\mu(B)\notin\eta(B)$ or $\lambda\in(0,\mu(B)]$ if
$\mu(B)\in\eta(B)$. 

If $\lambda\in\eta(B)$ where $\eta(B)=\{\mu(B)\}$ if $\nu(B)=\mu(B)$ or $\eta(B)$
an interval whose interior is $(\nu(B), \mu(B))$, then there is 
an irreducible  nonnegative matrix $E$ such that $\Aux(E)= B$ with $\lambda=\rho(A)$. 

In the remaining case $\lambda\leq \nu(B)$ we will construct a row uniform matrix $H$ so that 
$\nu(H)\leq\lambda\leq \mu(H)$. 
Since $B$ has at least two cycles 
and it is irreducible, there exists a row with index belonging to one of those
cycles and with at least two nonzero elements one of which must be on that cycle. 
Let $t$ be the index of such row. Consider a cycle $\alpha$
going through that row, with cycle mean $c$ and length $\ell$. If we have
$c\leq\lambda$, it follows that $\nu(B)\leq\lambda\leq \mu(B)$
and we select $H=B$.  If $c>\lambda$
% and find among them one with the smallest length.  
then we multiply all entries of row $t$ by $z$ such that
$c\cdot z^{1/\ell}=\lambda$.  Let $H$ be the resulting matrix, so we have $0<\nu(H)\leq\lambda\leq \mu(H)$.

If $\nu(H)<\mu(H)\leq \nu(B)<\mu(B)$ and $\lambda=\mu(H)$ then $\mu(H)$ is the 
new mean value of the cycle $\alpha$, which previously had $c>\lambda$.
In this case, the corresponding factor $z<1$ can be slightly increased so that 
$\nu(H)<\lambda<\mu(H)$ is satisfied.
If $\nu(H)<\mu(H)$ and $\lambda=\nu(H)$, then multiplying the row $t$  
by a value $1-\epsilon$ for small enough $\epsilon$ we can also ensure that
$\nu(H)<\lambda<\mu(H)$. 

Thus we can assume that $\nu(H)=\lambda=\mu(H)$ or 
$\nu(H)<\lambda<\mu(H)$, where $H$ is obtained from $B$ by multiplying the row $t$ with at least two 
nonzero entries by a nonnegative scalar $z\leq 1$.
Then by Theorem~\ref{t:mainres}, there is a
nonnegative matrix $E$ with an eigenvector $v$ such that $Ev=\lambda v$ 
and $\Aux(E)= H$, where row $t$ has at least two nonzero entries that we denote by 
$e_{tk}$ and $e_{tl}$. Since $E$ is irreducible, all components of $v$ are positive.
We now modify row $t$ of $E$ to form a matrix $C$ such that 
$\Aux(C)= B$ and $Cv=\lambda v$. Let $x$ be such that 
\begin{equation}
%\begin{split}
\sum_{s\neq k,l}e_{ts}+\sqrt{e^2_{tk}+(x/v_k)^2} +\sqrt{e^2_{tl}+(x/v_{l})^2}=b_{tk}. 
%&\sum_{s\neq k_i,l_i}e_{t_is}+\sqrt{e^2_{t_ik_i}+x^2} =b_{t_ik_i},\ \text{if} v_{k_i}=v_{l_i}=0.
%\end{split}
\end{equation}
It can be observed that this equation can be explicitly resolved with respect to $x$.

%In the following description of the matrix $C=(c_{rs})$ the imaginary number ``$i$''
%is denoted by $\Im$:
%\newline
 \begin{math}
	 c_{rs}=
	 \begin{cases}
  e_{tk}-\Im (x/v_{k}), & \text{if $r=t,\,s=k$;}\\
  e_{tl}+ \Im (x/v_l), & \text{if $r=t,\,s=l$;}\\
%e_{t_ik_i}+ x_i\Im, & \text{if $r=t_i,\,s=k_i\,v_{k_i}=v_{l_i}=0$;}\\  
e_{rs}, & \text{otherwise}.
  	\end{cases}
\end{math}

Then $\Aux(|C|)=B$ and $Cv=\lambda v$, so $\lambda$ is an eigenvalue of $C$.
The claim follows.

%We now show that 
%$\lambda=\rho(C)$ by showing that any eigenvalue of C is also an eigenvalue of E. Given an eigenvalue $\beta$ the equations that the eigenvector must satisfy are given by $CV=\beta V$. But  V satisfies  $\sum e_{k,i} v_i =\beta v_k$ for $k\neq t$. This is a set of $n-1$ equations and since we can normalize so that  $v_1=1$ there are $n-1$ unknowns. \textsl{\textbf{Need Help to make this into a proof. It probably works for the $Aux(|A|)\preceq B$ since the equations are formed by two $\gamma-sunflowers.$}}
%\newline 
%\par 

%Thus $\lambda=\rho(E)=\rho(C)$ and $\eta(B)\subseteq (0,\mu(B))$. By the definition of $\eta(B)$ $0\notin\eta(B)$. Since A is irreducible then if $\mu(B)>\nu(B)$ we have that for $\lambda$ a non-negative  eigenvalue of A  $\lambda\leq\rho(A)\leq\rho(|A|)$. By theorem \ref{range} (2) we have that $\rho(|A|)\in(0,\mu(B))$ and (ii) is established. 
%\newline
%\par
%If $\mu(B)=\nu(B)$ and if there are two cycles in B then $(0,\mu(B))\subseteq\eta(B)$ and as a consquence of theorem \ref{range} we have that $\{\mu(B)\}\subseteq\eta(B)$ thus  $\eta(B)= (0,\mu(B)]$ and (i) is established. 
\end{proof}

We call a class of complex matrices {\em regular} if all matrices in the class are nonsingular.

\begin{corollary}
\label{c:crhorange-irred1}
Let $B$ be an irreducible row uniform nonnegative multicyclic matrix with all diagonal elements equal to $0$. Let $\Gamma(B)$ consist of all complex matrices $I-A$ with $\Aux(|A|) = B$. 
\begin{itemize}
\item[{\rm (i)}] 
If  $\mu(B) < 1$ then  $\Gamma(B)$ contains only regular matrices.
\item[{\rm (ii)}] 
 If $\mu(B) = 1$ then  $\Gamma(B)$ contains only regular matrices\\ if and only if $\nu(B) < 1$.
\item[{\rm (iii}] 
If  $\mu(B) > 1$ then  $\Gamma(B)$ contains a singular matrix.
\end{itemize}
\end{corollary}
\begin{proof}
$\Gamma(B)$ contains a singular matrix if and only if $1 \in \sigma(B)$. By Theorem~\ref{t:crhorange-irred}
this happens if and only if either $\mu(B)>1$ or $\mu(B)=1=\nu(B)$. 
This establishes all the claims.
\end{proof}

\begin{remark} {\rm
As noted in the abstract and introduction of \cite{BBCH},  
the theorems in that paper imply Brualdi's \cite{B} conditions for 
the non-singularlty of matrices and show that they are sharp. 
There is no essential difference or simplification in assuming that the 
main diagonal of the matrices considered there is the identity, and in that case the 
spectral content of  \cite{BBCH} Theorems 1.1  -- 1.4 is recaptured by Corollary \ref{c:crhorange-irred1} via standard 
Ger\v{s}gorin theory, e.g.\cite{T}.  More precisely, Corollary \ref{c:crhorange-irred1}(i) corresponds to Theorem 1.1 of \cite{BBCH},  \ref{c:crhorange-irred1}(ii)  
corresponds to Theorems 1.2 and 1.3, and \ref{c:crhorange-irred1}(iii) corresponds to Theorem 1.4.
}
\end{remark}

For $A\in\Rpnn$, index set $K$ and row uniform matrix $B$ we write
$\Aux(A)\lneq^K B$ when the following conditions hold.
\begin{itemize}
\item[(a)] For $\Aux(A)=\Tilde{B}=(\tilde{b}_{ij})$ we have 
$\tilde{b}_{ij}=0\Leftrightarrow b_{ij}=0$ for all $i,j$.  
\item[(b)] For all $i\in K$ we have $\tilde{b}_{ij}<b_{ij}$ for all $j$ where $b_{ij}>0$. 
\item[(c)] For all $i\notin K$ and all $j$ we have $\tilde{b}_{ij}=b_{ij}$.
\end{itemize}

We will also need the following variation of Definition~\ref{def:sigma}.

\begin{definition}
\label{def:tildesigma}
For $B$ a row uniform matrix,  let $\tilde{\sigma}_K(B)$ denote the set 
$\tilde{\sigma}_K(B)=\left\{\lambda \colon\exists A\in\C^{n\times n},\ \Aux(|A|)\lneq^K B,\ \det(A-\lambda I)=0 \right\}$.
\end{definition}

The following corollary of Theorem~\ref{t:crhorange-irred} is immediate.

\begin{corollary}
\label{c:crhorange-irred}
Let $B$ be a row uniform nonnegative irreducible matrix. Then for any non-empty index set $K$,\\
$\tilde{\sigma}_K(B)\backslash\{0\}= \{ s\colon 0<|s|<\mu(B)\}$.
\end{corollary}
\begin{proof} Let us analyze the following three cases.

Case 1: $\mu(B)>\nu(B)$. There exists an $A$ such that 
 $\mu(\Aux(|A|)$ is arbitrarily close to $\mu(B)$ and
$\nu(\Aux(|A|)<\mu(\Aux(|A|) $, and
for each $A$ with  $\Aux(|A|)\lneq^K B$
we have $\nu(\Aux(|A|)<\mu(B)$.

Case 2: $\mu(B)=\nu(B)$ and each cycle contains an index from $K$. In this case 
$\mu(\Aux(|A|)$, where $\Aux(|A|)\lneq^K B$, 
assumes all values in $(0,\mu(B))$.

Case 3: $\mu(B)=\nu(B)$ and there is a cycle avoiding
the nodes with indices in $K$. In this case $\mu(\Aux(|A|))=\mu(B)$ for all $A$ with
$\Aux(A)\lneq^K B$, but $\nu(\Aux(|A|))<\mu(\Aux(|A|))$ for all such matrices. 

In all three cases we obtain the claim by applying Theorem~\ref{t:crhorange-irred}
to all $\Aux(|A|)$ satisfying $\Aux(|A|)\lneq^K B$.
\end{proof}

We are now ready to deal with the general reducible case.

\begin{theorem}
\label{t:crhorange-red}
Let $B$ be a row uniform nonnegative matrix, and let 
\begin{equation}
\begin{split}
& \Tilde{M}(B):= \max\{\mu(B_i)\ \text{where}\\
& B_i\ \text{is a transient class or a final multicyclic class of}\ B\}.
\end{split}
\end{equation}
Then
\begin{itemize}
\item[(i)] If $\Tilde{M}(B)$ is attained at some 
final multicyclic class $B_s$ with $\nu(B_s)=\mu(B_s)$ then
\begin{equation}
\begin{split}
&\sigma(B)\backslash\{0\}=\{s\colon 0<|s|\leq\Tilde{M}(B)\}\cup
 \\
& \cup_i \{s\colon |s|=\mu(B_i),\ 
\text{$B_i$ is a final unicyclic class and 
$\mu(B_i)>\Tilde{M}(B)$} \}.
\end{split}
\end{equation}

\item[(ii)] Otherwise,
\begin{equation}
\begin{split}
&\sigma(B)\backslash\{0\}=\{s\colon 0<|s|<\Tilde{M}(B)\}\cup
\\
&\cup_i \{s\colon |s|=\mu(B_i),\  
\text{$B_i$ is a final unicyclic class and $\mu(B_i)\geq\Tilde{M}(B)$}\}.
\end{split}
\end{equation}

\end{itemize}
\end{theorem}
\begin{proof}
It is known that $\lambda$ is an eigenvalue of a matrix $A\in\C^{n\times n}$ if and only if $\det(A-\lambda I)=0$, which
implies that the spectrum of $A\in\C^{n\times n}$ (i.e., the set of eigenvalues of $A$) 
is the union of spectra of its nontrivial classes in the Frobenius normal form.
 Furthermore, if a principal submatrix $A_s$ corresponds to a 
transient 
class then it can be any matrix satisfying $\Aux(|A_s|)\lneq^{K_s} B_s$,
where $K_s$ is the (non-empty) set of indices of all
nodes in this transient class that have a connection to another class.
Observe that the entries in different rows of matrices with the same $\Aux(|A|)$ vary 
independently and hence the same is true about the sets of rows belonging
to different classes.
Therefore 
$\sigma(B)\backslash\{0\}$ can be found as union of $\sigma(B_i)\backslash\{0\}$ over all final classes $B_i$
and $\Tilde{\sigma}_{K_s}(B_s)\backslash\{0\}$ over 
all transient classes $B_s$, for some non-empty index sets $K_s$.
Using Theorem~\ref{t:crhorange-irred} and Corollary~\ref{c:crhorange-irred} and taking the above mentioned union, it can be verified that 
$\sigma(B)\backslash\{0\}$ is as claimed.
\end{proof}

{\bf Example.} 
To illustrate the last theorem, let us consider the following row uniform matrices:
\begin{equation*}
B=
\begin{pmatrix}
5 & 0 & 0 & 0 & 0\\
4 & 0 & 4 & 0 & 0\\
0 & 4 & 0 & 0 & 0\\
3 & 0 & 0 & 3 & 3\\
0 & 3 & 0 & 3 & 3
\end{pmatrix},\quad 
C=
\begin{pmatrix}
5 & 0 & 0 & 0 & 0\\
0 & 0 & 4 & 0 & 0\\
0 & 4 & 0 & 0 & 0\\
0 & 0 & 0 & 3 & 3\\
0 & 0 & 0 & 3 & 3
\end{pmatrix}
\end{equation*}
That is, $C$ is formed from $B$ by cutting all connections between the classes. 

The moduli of the eigenvalues in $\sigma(B)$ assume all the values in 
$(0,4)\cup\{5\}$. 
Note that $\tilde{M}(B) =\max\{3,4\}$, but $4\notin\sigma(B)$ 
because the class extracted from rows and columns $2$ and $3$ is 
transient ($b_{12}>0$). Therefore  
condition (ii) of the Theorem~\ref{t:crhorange-red} is used in 
computing $\sigma(B)$.

\if{
The moduli of eigenvalues in $\sigma(B)$ assume all values in $(0,4)\cup\{5\}$
Note that $\Tilde{M}(B)=\max(3,4)=4$ and that $4$ does not
belong to the set: although the second class of $B$ (the block extracted from 
rows and columns $2$ and $3$) is unicyclic, it is also transient. Thus its maximum
cycle mean takes part in the computation of $\Tilde{M}(B)$ but in turn, it also
prevents $\Tilde{M}(B)$ from being attained.
}\fi

The moduli of eigenvalues in $\sigma(C)$ assume all values in $(0,3]\cup
\{4\}\cup\{5\}$. Here $\Tilde{M}(C)=3$, which is the maximum cycle mean of 
the only final class which is multicyclic. As the means of all cycles in that class are
equal to each other, the value of $\Tilde{M}(C)$ belongs to $\sigma(C)$.

\subsection{Camion-Hoffman theorem}
\label{ss:camhoff}

We now will apply Theorem~\ref{t:perronvis} and 
Theorem~\ref{t:crhorange-red} to provide a new proof for a theorem of Camion and Hoffman~\cite{CH}. 

Let us first recall the following known facts and a definition:

\begin{lemma}[{\cite{CH}}]
\label{l:camhoff}
Let $a_1,\ldots a_n$ be nonnegative numbers such that each number does not exceed the sum of other numbers.
Then there exist complex numbers $c_1,\ldots, c_n$ such that
$|c_i|=a_i$ for $i=1,\ldots,n$ and $c_1+\ldots+c_n=0$.
\end{lemma}

\begin{corollary}[{\cite{CH}}]
\label{c:camhoff}
Let the entries of $A=(a_{ij})\in\Rpnn$ satisfy $a_{ii}=1$, $\sum_{j\neq i} a_{ij}\geq 1$ for all $i$
and $a_{ij}\leq 1$ for all $i,j$.
Then there exists a complex matrix $C$ with $|C|=A$ and $\det(C)=0$.
\end{corollary}
\begin{proof}
Since the condition of Lemma~\ref{l:camhoff} are satisfied for $a_{i1},\ldots,a_{in}$ 
for all $i$, there exists a complex matrix $C$ with $|C|=A$ such that $\sum_j c_{ij}=0$
for all $i$. 
This implies $\det(C)=0$.
\end{proof}

\begin{definition}
A matrix $A=(a_{ij})\in\Rpnn$ is called strictly diagonally dominant if $a_{ii}>\sum_{j\neq i} a_{ij}$ for 
all $i$.
\end{definition}

We will investigate the following matrix class:

\begin{definition}
 For $A\in\Rpnn$ define $\Omega(A)=\{E:|e_{ij}|=a_{ij}\ 1\leq i,j\leq n\}$.
\end{definition}

\begin{theorem}[Camion-Hoffman]
 For $A\in\Rpnn$ the following are equivalent:
\begin{itemize}
\item[(i)] $\Omega(A)$ does not contain a singular matrix;
\item[(ii)] There exists a permutation matrix $P$ and a diagonal matrix $D$ such that
$PAD$ is strictly diagonally dominant;
\item[(iii)] There exists a 
permutation matrix $P$ and nonsingular diagonal matrices $D_1, D_2$ such that all diagonal entries of
$D_1PAD_2$ are equal to $1$ and $\mu(\Aux(D_1PAD_2-I))<1$. 
\end{itemize}
\end{theorem}
\begin{proof}
(i)$\Rightarrow$ (ii): Assume that $\Omega(A)$ is regular. 
Let $P$ be a permutation matrix such that the 
diagonal product of $PA$ is greater then or equal to any 
generalized diagonal product of $A$. Since $A$ is nonsingular the diagonal elements of $E=PA$ are 
nonzero.  
Let $D$ be the diagonal matrix with entries equal 
to the inverse of the corresponding diagonal elements of $PA$.
Since all 
diagonal entries of $PAD$ are equal to $1$, 
for any cycle $\alpha$ we can find a generalized diagonal product of $PAD$ 
equal to the product of the 
entries of $\alpha$. 
Since any generalized diagonal product of $PAD$ is less
than or equal to $1$, it follows that $\mu(PAD)=1$.  %As $\mu(PAD)=1$, there exists a nonsingular diagonal matrix $X\geq 0$  
%such that  $F=X^{-1}PAD X$ satisfying $0\leq f_{ij}\leq 1\ for\ 1\leq i,j \leq n$. 

We will now establish that $\rho(PAD-I)<1$. The proof is by contradiction.
Assume that $\rho(PAD-I)\geq 1$.
Then $PAD-I$ has a class $B$ such that $\rho(B)>1$ and for all $i$ we have $b_{ii}=0$. 
Since $\mu(PAD-I)\leq 1$ we also have $\mu(B)\leq 1$. Applying Theorem~\ref{t:perronvis} 
to $B$, we obtain a diagonal nonnegative matrix $Y$ such that matrix
$E:=Y^{-1}(B+I)Y$ has entries satisfying
$0\leq e_{ij}\leq 1$ and $e_{ii}=1$ for all $i,j,$ and $\sum_{k\neq i} e_{ik}\geq 1$ for all $i$.
By Corollary~\ref{c:camhoff} there is a matrix $H=(h_{ij})$ with complex entries satisfying 
$|H|=E$ and $\det(H)=0$. Replacing the class $B+I$ in $F$ by $YHY^{-1}$ we obtain a matrix
$G$ with $\det(G)=0$ and $|G|\in\Omega(PAD)$. As $P$ is a permutation matrix and $D$ diagonal,
there is a bijective correspondence between $\Omega(PAD)$ and $\Omega(A)$ in which the singularity
and nonsingularity are preserved. This contradicts that $\Omega(A)$ does not 
contain a singular matrix and hence $\rho(PAD-I)<1$.

Since $\rho(PAD-I)<1$, there exists a diagonal matrix 
$Z$ such that $Z^{-1}(PAD-I)Z$ has all row sums strictly less than $1$, 
% [as on p.26, l-13],
 see \cite[Chapter 6]{BP} or \cite{S1} for a detailed argument.
(Such a diagonal matrix $Z$ can be constructed using Perron
eigenvectors of nontrivial classes.) As all row sums in the
 matrix $Z^{-1}(PAD-I)Z=Z^{-1}PADZ-I$ 
are strictly less than $1$, it follows that
the matrix $PADZ$ is strictly diagonally dominant, with $P$ a permutation matrix
and $DZ$ a diagonal matrix, as required.

\if{
Thus we have $\rho(PAD-I)<1$, meaning that the Perron roots of all 
nontrivial classes of $PAD-I$ are less than $1$. 
We next argue that there exists a diagonal matrix $Z$ such that $Z^{-1}(PAD-I)Z$ has all row sums
strictly less than $1$.  In fact, the argument is standard and can be skipped, see
for instance~\cite{BP} or \cite{S1}.  

Let $B_i$ for $i=1,\ldots, m$ be the
classes of the (lower triangular) Frobenius normal form of $PAD-I$ (arranged from the north-western corner
to the south-eastern corner of $PAD$), and let $B_{ij}$ for $i\neq j\in\{1,\ldots,m\}$ be
the corresponding off-diagonal blocks of $PAD-I$. If $i<j$ then $B_{ij}=0$.  
  As $\rho(B_i)<1$, for every such block there exists a 
nonsingular (nonnegative) diagonal matrix $Z_i$ such that $Z_i^{-1}B_iZ_i$ has all row sums less than $1$ (using the 
Perron vector for each block).
Here we also include the trivial classes $B_i$ for which such scaling is trivial (take any positive number).
Let $\delta$ be the smallest difference between $1$ and a row sum in one of the $Z_i$.  

Combining $Z_i$ for $i=1,\ldots,m$ we can build a diagonal matrix $Z$ 
 such that all row sums of 
$Z^{-1}(PAD-I)Z$ are less than $1$. Indeed, each diagonal matrix $Z_i$ can be multiplied by any scalar: 
this will not affect the entries of $Z_i^{-1} B_i Z_i$. 
Let $Z$ be the diagonal matrix  made from blocks 
$a_i Z_i$ for $i=1,\ldots,m$ . Here, scalars $a_1,\ldots, a_m$ can be selected in such a way that for
any $i=2,\ldots,m$, the total sum of all entries in the blocks $a_i^{-1}Z_i^{-1} B_{ij}Z_j a_j$ over 
$j=1,\ldots, i-1$ does not 
exceed $\delta$. Omitting the details, let us mention that this can be achieved 
by selecting $a_1,\ldots, a_m$ in such a way that 
$a_1\ll a_2\ll\ldots\ll a_m$. 

Thus there exists a diagonal matrix $Z$ such that all row sums in 
$Z^{-1}(PAD-I)Z=Z^{-1}PADZ-I$ 
are strictly less than $1$.
}\fi

(ii)$\Rightarrow$ (iii)
If $PAD$ is strictly diagonally dominant then there is a digonal matrix $D_1$ such that 
the diagonal entries of $D_1PAD$ are equal to $1$ and the row sums of
$D_1PAD-I$ are strictly less than $1$.
 As each entry in $\Aux(D_1PAD-I)$ is strictly less than $1$, we also have
$\mu(\Aux(D_1PAD-I))<1$ as claimed.

(iii)$\Rightarrow$ (i): The proof is by contradiction. Assume that (iii) holds
but (i) does not hold. That is, assume that  
there exists a permutation matrix 
$P$ and nonsingular diagonal matrices $D_1,D_2$ 
such that $\mu(\Aux(D_1PAD_2-I))<1$, and  that (in contradiction with (i)) there exists 
$C\in\Omega(A)$ with $\det(C)=0$. Then $\mu(\Aux(D_1P|C|D_2-I))=\mu(\Aux(D_1PAD_2-I))<1$, and by Theorem~\ref{t:crhorange-red}
we have  $1\notin\sigma(\Aux(D_1PAD_2-I))$. However, we have $\det(D_1PCD_2)=0$, and we can multiply the
rows of $D_1PCD_2$ by some complex numbers with moduli $1$ to obtain a matrix with zero determinant
and with all diagonal entries equal to $-1$. Adding the identity matrix to this matrix we obtain a 
matrix in the class $\Omega(D_1PAD_2-I)$, for which $1$ is an eigenvalue. 
The set of eigenvalues of matrices in
$\Omega(D_1PAD_2-I)$ is a subset of $\sigma(\Aux(D_1PAD_2-I))$, so 
$1\in \sigma(\Aux(D_1PAD_2-I))$, a contradiction.   
\end{proof}
%\end{proof}

Let us also reformulate the Camion-Hoffman theorem in terms of $M$-matrices
and comparison matrices.
Recall that a real matrix $B$ is a nonsingular $M$-matrix if $B=\rho I- C$ where
$C$ is a nonnegative matrix and the Perron root of $C$ is strictly less than $\rho$ 
(see~\cite{BP} for many other equivalent definitions). For a nonnegative matrix $A\in\Rpnn$,
its comparison matrix $E=\comp(A)$ has entries $e_{ii}=a_{ii}$ for $i=1,\ldots,n$ and
$e_{ij}=-a_{ij}$ for $i\neq j$.  

\begin{theorem}
For a nonnegative matrix $A$, the following are equivalent:
\begin{itemize}
\item[(i)] $\Omega(A)$ does not contain a singular matrix,
\item[(ii)] For $P$ a permutation matrix corresponding
to the greatest generalized diagonal product of $A$, the matrix $\comp(PA)$ is
a nonsingular $M$-matrix.  
\end{itemize}
\end{theorem}

\end{document}